\documentclass[leqno,11pt]{amsart}

\usepackage{setspace}
\usepackage[english]{babel}
\usepackage{bussproofs}
\usepackage{csquotes}
\usepackage{mathtools}
\usepackage{mathrsfs}
\usepackage{latexsym}
\usepackage{enumitem}
\usepackage{amsthm}
\usepackage{amssymb}
\DeclareMathAlphabet{\mathpzc}{OT1}{pzc}{m}{it}
\usepackage[latin1]{inputenc}
\usepackage{afterpage}

\usepackage{bbm}

\newtheorem{theorem}{Theorem}[section]
\newtheorem*{theorem*}{Theorem}

\newtheorem{lemma}[theorem]{Lemma}

\newtheorem{corollary}[theorem]{Corollary}

\newtheorem{fact}[theorem]{Fact}
\newtheorem{claim}[theorem]{Claim}
\theoremstyle{definition}
\newtheorem{definition}[theorem]{Definition}

\theoremstyle{remark}

\newtheorem{notation}{Notation}
\newtheorem{question}{Question}

\def\hook{\upharpoonright}
\def\forces{\Vdash}

\def\Me{\mathcal M}

\def\ZFC{\mathsf{ZFC}}

\def\baire{\omega^\omega}

\def\GCH {\mathsf{GCH}}
\def\CH {\mathsf{CH}}
\def\Q{\mathbb Q}
\def\P{\mathbb P}

\def\mfa{\mathfrak{a}}

\def\mfi{\mathfrak{i}}

\def\cof{{\rm cof}}

\def\mfu{\mathfrak{u}}

\usepackage[margin=1.3in]{geometry}

\begin{document}

\title{Cohen Preservation and Independence}
\author[Fischer]{Vera Fischer}
\address[V. ~Fischer]{Institut f\"{u}r Mathematik, Kurt G\"odel Research Center, Universit\"{a}t Wien, Kolingasse 14-16, 1090 Wien, AUSTRIA}
\email{vera.fischer@univie.ac.at}

\author[Switzer]{Corey Bacal Switzer}
\address[C.~B.~Switzer]{Institut f\"{u}r Mathematik, Kurt G\"odel Research Center, Universit\"{a}t Wien, Kolingasse 14-16, 1090 Wien, AUSTRIA}
\email{corey.bacal.switzer@univie.ac.at}

\thanks{\emph{Acknowledgements:} The authors would like to thank the
Austrian Science Fund (FWF) for the generous support through grant number Y1012-N35.}

\date{}

\maketitle

\begin{abstract}
We provide a general preservation theorem for preserving selective independent families along countable support iterations. The theorem gives a general framework for a number of results in the literature concerning models in which the independence number $\mathfrak{i}$ is strictly below  $\mathfrak{c}$, including iterations of Sacks forcing, Miller partition forcing, $h$-perfect tree forcings, coding with perfect trees. Moreover, applying the theorem, we show that $\mathfrak{i} = \aleph_1$ in the Miller Lite model. An important aspect of the preservation theorem is the notion of ``Cohen preservation", which we discuss in detail.
\end{abstract}

\section{Introduction}
A family $\mathcal I \subseteq [\omega]^\omega$ is said to be {\em independent} if for all finite subsets $X_0, ..., X_{n-1} \in \mathcal I$ and all $g:n \to 2$ we have that $X^{g(0)}_0 \cap ... \cap X^{g(n-1)}_{n-1}$ is infinite where $X_i^{g(i)}$ means $X_i$ if $g(i) = 1$ and $\omega \setminus X_i$ if $g(i) = 0$. Such a set is said to be {\em maximal} if it is not properly contained in any other independent set. The {\em independence number} $\mfi$ is the least size of a maximal independent family. The cardinal characteristics $\mfi$ is difficult to control and there remain many interesting, long standing open questions, like the consistency of $\mfi < \mfa$ where $\mfa$ is the almost disjointness number\footnote{See the appendix of \cite{CFGS21} for a discussion of this problem.}, or the consistency of $\hbox{cof}(\mfi)=\omega$. While $\mfi$ has no known upper bound, besides $\mathfrak{c}=2^{\aleph_0}$, both the reaping number $\mathfrak{r}$ and the dominating numbers $\mathfrak{d}$
(see~\cite{BlassHB} for definitions) are $\ZFC$ provably less than or equal to $\mfi$. Thus preserving $\mfi$ small requires preserving the smallness of several other cardinal characteristics simultaneously.

One of the first breakthroughs in studying $\mfi$ can be found in \cite{Sh92} where the consistency of $\mfi < \mfu$ is established\footnote{The cardinal $\mfu$, the {\em ultrafilter number} is the least size of an ultrafilter base on $\omega$.}. There, under $\CH$, a special independent family, now known as a {\em selective} independent family is constructed and used as a witness to $\mathfrak{i}=\aleph_1$. A selective independent family (see Definition~\ref{def.selective.independent}) is a densely maximal independent family (see Definition~\ref{densemax}), with the property that a naturally associated filter to the family, namely its density filter (see Definition~\ref{density.filter}), is Ramsey.  

Selective independence is one of the main tools in providing models of $\mfi < 2^{\aleph_0}$, as selective independent families are indestructible under the countable support iterations of: Sacks forcing (see~\cite[Theorem 4.6]{DefMIF} or \cite[Corollary 37]{FM19}); Miller partition forcing (see~\cite{CFGS21}); $h$-perfect tree forcing (see~\cite{CS}); Shelah's poset $\mathbb{Q}_\mathcal{I}$ from~\cite{Sh92}; coding with perfect trees as in~\cite{BFS21}. In each of those cases the proof of preservation of selectivity runs along similar lines: One relies on the Sacks property, or on a property called {\em Cohen preservation} (which follows from the Sacks property, see Lemma~\ref{Sacks.implies.Cohen} or~\cite[Proposition 14]{free_seq}) alongside $\baire$-boundedness, to ensure the  preservation at the limit stages via a  preservation result of Shelah (see Theorem~\ref{Shelah_preservation1}), while a separate and partial order  specific argument is given at successor stages. The proofs at the successor stages seem just different enough, that it is not quite clear how to unify them into a single, succinct preservation theorem. 

Recall that a partial order is said to be Cohen preserving, if every new dense open subset of $2^{<\omega}$ contains an old dense subset (see Definition~\ref{def.cohen.preservation}). It turns out that the use of Cohen preservation alongside $\baire$-boundedness is somewhat illusory as, as we show below, Cohen preservation in fact implies $\baire$-boundedness (at least for proper forcing notions or under certain cardinal characteristic constellations, see Theorem~\ref{thm2}):

\begin{theorem*}
	If $\P$ is a proper Cohen preserving forcing notion, then $\P$ is $\baire$-bounding.
	\label{thm2}
\end{theorem*}

Moreover, we establish the following general preservation theorem (see Theorem~\ref{mainthm11}), providing in particular a uniform framework for the preservation results listed above, as well as any countable support mixture of the above mentioned posets (see~Fact~\ref{fact.all}):

\begin{theorem*}
($\CH$) Let $\delta$ be an ordinal, $\mathcal I$ a selective independent family and let $\langle \P_{\alpha}, \dot{\Q}_\alpha\; | \; \alpha < \delta\rangle$ be a countable support iteration of forcing notions so that for each $\alpha < \delta$ we have $$\forces_\alpha``\dot{\Q}_\alpha\hbox{ is proper and Cohen preserving}".$$ 
If for every $\alpha < \delta$, 
$$\forces_\alpha``\dot{\Q}_\alpha\hbox{ preserves the dense maximality of }\mathcal I",$$ 
then $\P_\delta$ preserves that $\mathcal I$ is selective and in particular maximal.
\end{theorem*}

Apart from offering this general preservation theorem, as an application, we obtain the following new cardinal characteristics constellation, see Theorem~\ref{millerlite2} and Corollary~\ref{millerlite2.corollary}:

\begin{theorem*}\label{millerlite1}
Miller lite forcing, $\mathbb{ML}$, preserves dense maximality of  selective independent families, it is Cohen preserving and proper. Thus $$\mfi = \mathfrak{hm} < \mathfrak{l}_{n, \omega}$$ is consistent and moreover the witness for $\mfi$ can be taken to be co-analytic.
\end{theorem*}

Miller Lite forcing is first investigated by Geschke in \cite{Ges06}, in relation to continuous colorings of the reals. The definition of the {\em localization numbers} $\mathfrak{l}_{n, \omega}$ and the cardinal $\mathfrak{hm}$ are omitted as it they are not used in this paper. See \cite{Ges06} for details.

The rest of the paper is organized as follows. In the next section we review the basics of Cohen preservation and prove a number of equivalent formulations which, in particular, imply Theorem~\ref{thm2}. In Section 3 we prove Theorem~\ref{mainthm11}. In section 4 we discuss applications and prove Theorem~\ref{millerlite2} and Corollary~\ref{millerlite2.corollary}. The final section includes some open questions and remarks. Our notation is mostly standard, we refer the reader to~\cite{BarJu95},~\cite{BlassHB} and~\cite{Hal17} for all undefined notions concerning forcing and cardinal characteristics. 

\section{Cohen Preservation}

Recall the notion of {\em Cohen preservation}. 

\begin{definition}\label{def.cohen.preservation}
Let $\P$ be a forcing notion. We say that $\P$ is {\em Cohen preserving} if every every new dense open subset of $2^{{<}\omega}$ (or, equivalently of $\omega^{{<}\omega}$) contains an old dense subset. More formally, $\P$ is Cohen preserving if for all $p \in \P$ and all $\P$-names $\dot{D}$ so that $p \forces$ ``$\dot{D} \subseteq2^{{<}\omega}$ is dense open" there is a dense $E \subseteq 2^{{<}\omega}$ in the ground model and a $q \leq_\P p$ so that $q \forces \check{E} \subseteq \dot{D}$.
\end{definition}

Another way of expressing the above is to say that every $\P$-name $\dot{D}$ for a dense (not necessarily open) subset of $2^{<\omega}$ and every $p \in \P$ there is a $q \leq p$ and a dense $E \in V$ so that $q$ forces that $E$ {\em refines} $\dot{D}$ i.e. for every $t \in \dot{D}$ there is an $s \supseteq t$ with $s \in E$. Being Cohen preserving is preserved by countable support iterations of proper forcing notions.

\begin{theorem}[Shelah, See Conclusion 2.15D, pg. 305 of \cite{PIP}, see also \cite{FM19}, Theorem 27]
If $\delta$ is an ordinal and $\langle \Q_\alpha, \dot{\mathbb R}_\alpha \; | \; \alpha < \delta\rangle$ is a countable support iteration of forcing notions so that for each $\alpha < \delta$ we have $\forces_\alpha$``$\dot{\mathbb R}_\alpha$ is proper and Cohen preserving" then $\Q_\delta$ is proper and Cohen preserving.
\label{Shelah_preservation2}
\end{theorem}

The notion is also closely related to the cardinal invariant $\cof(\Me)$, the cofinality of the meager ideal. It is for this reason that it is important in preserving small witnesses to $\mfi$ since $\cof(\Me) \leq \mfi$ in $\ZFC$ (see \cite[Theorem 3.6]{BHHH04}). 
A detailed proof of the following Lemma can be found in~\cite[Proposition 14]{free_seq}

\begin{lemma}\label{Sacks.implies.Cohen} If a forcing notion $\mathbb{P}$ has the Sacks property, then $\mathbb{P}$ is Cohen preserving.	
\end{lemma}

Moreover, we obtain the following equivalent characterisations of Cohen preservation.

\begin{theorem}\label{Cohenequiv}
Let $\P$ be a proper forcing notion. The following are equivalent.
\begin{enumerate}
\item
$\P$ is Cohen preserving.
\item
If $\dot{A}$ is a $\P$-name, $p$ and condition and $p \forces$ ``$\dot{A}$ is a Borel nowhere dense set" then there is a $q \leq p$ and a Borel nowhere dense set $A_q$ so that $q \forces \dot{A} \subseteq \check{A}_q$ with $\check{A}_q$ the name for the reiterpretation of $A_q$ in the extension. 
\item
For all $\P$-names $\dot{A}$ and conditions $p$ so that $p \forces$ ``$\dot{A}$ is a Borel meager set" there is a $q \leq p$ and a Borel meager set $A_q$ so that $q \forces \dot{A} \subseteq \check{A}_q$.
\item
$\P$ is $\baire$-bounding and for every condition $p \in \P$, every sufficiently large $H_\theta$, every countable $M \prec H_\theta$ with $p, \P \in M$ and every $c \in 2^\omega$ which is Cohen over $M$, there is a $q_c \leq p$ forcing that $c$ is Cohen over $M[\dot{G}]$ where $\dot{G}$ is the canonical name for a $\P$-generic filter. 
\end{enumerate}
\end{theorem}

\begin{proof}
We will show $(1)$ implies $(2)$, $(2)$ implies $(3)$, $(3)$ implies $(4)$ and $(4)$ implies $(1)$. The implication $(1)$ to $(2)$ follows from the observation that $D \subseteq 2^\omega$ is dense open in the standard topology on $2^\omega$ if and only if $D' =\{s \in 2^{<\omega} \; | \; [s] \subseteq D\}$ is dense in $2^{<\omega}$ (in the forcing sense). Let us begin with this point. First note that if $D$ is dense open, then for every $s \in 2^{<\omega}$ there is a $t \supseteq s$ so that $[t] \subseteq D$ since $D \cap [s] \neq \emptyset$ and, any element $x \in D$ is in an open neighborhood (and hence a basic open neighborhood) contained in $D$. Thus $D'$ is dense. Conversely, suppose $D'$ is dense. Observe that then $D = \bigcup_{s \in D'} [s]$. That $D \supseteq\bigcup_{s \in D'} [s]$ is by definition and the reverse inclusion follows from the fact that $D$ is open. But now, $\bigcup_{s \in D'} [s]$ is clearly dense open. 

Now suppose that $\P$ is Cohen preserving. If $\dot{A}$ names a Borel nowhere dense set, then by strengthening if necessary we can assume that $\dot{A}$ is closed nowhere dense since these are cofinal in the nowhere dense sets. Let $\dot{U}$ be the name for the dense open $2^\omega \setminus \dot{A}$. Then there is a name $\dot{U}'$ for the dense open subset of $2^{<\omega}$ as above and, by Cohen preservation there is in $V$ a $U_0'$ refining it. But then the dense open $U_0$ in $V$ is forced to be a subset of $\dot{U}$, so its complement contains $\dot{A}$ as needed.

To get from $(2)$ to $(3)$, let $\dot{A}$ name a Borel meager set. Then it is forced by some $p' \leq p$ that $\dot{A} \subseteq \bigcup_{n < \omega} \dot{C}_n$ for some name for a countable sequence of closed nowhere dense sets. By the countable covering property of properness plus $(2)$, we can find a family $\{D_n\}_{n < \omega}$ of nowhere dense closed sets and a $q \leq p'$ so that for every $n < \omega$ there is an $m < \omega$ with $q \forces \dot{C}_n \subseteq \check{D}_m$. It follows that $\bigcup_{n <\omega} D_n$ is the needed meager set.

Now assume $(3)$. First we show that $\P$ is $\baire$-bounding. This is essentially due to Pawlikowski and Rec\l aw, \cite[Lemma 1.13]{PR95}. In the proof of \cite[Lemma 1.13]{PR95} they show that there are ground model definable continuous functions $e, f:\baire \to \baire$ so that (provably in $\ZFC$) for all $x, y \in \baire$ we have that if $M_{e(x)} \subseteq M_y$ then $x \leq^* f(y)$, where for each $x \in \baire$ $M_x$ denotes the Borel meager set with code $x$ for some fixed coding (see \cite{PR95} for details). Now suppose that $\dot{x}$ is a $\P$-name for a real and work in some generic extension $V[G]$. Since item $(3)$ holds, we have that $M_{e(\dot{x}^G)} \subseteq M_y$ for some ground model $y$ and hence $\dot{x}^G \leq^* f(y)$. Since $f$ is a continuous function coded in the ground model, $f(y) \in V$, so $\P$ is $\baire$-bounding.

Still assuming $(3)$ we now show the second condition. Fix a condition $p \in \P$, a sufficiently large cardinal $\theta$ and a countable $M \prec H_\theta$ with $p, \P \in M$. Let $c \in 2^\omega$ be Cohen over $M$. If there is no $q \leq p$ forcing that $c$ is Cohen over $M[\dot{G}]$, then there is a $\P$-name $\dot{A}$ in $M$ for a Borel meager set and $p \forces \check{c} \in \dot{A}$. But by $(3)$ there is some $q \leq p$ and a ground model Borel meager set $A_q$ so that $q \forces \dot{A} \subseteq \check{A}_q$. Moreover, by elementarity such a $q$ and $Y$ can be found (coded) in $M$. It follows using Shoenfield absoluteness that $c \in A_q$ (in $M$ and in $V$). But this contradicts the fact that $c$ is Cohen over $M$.

Finally, we show that $(4)$ implies $(1)$. Let $\dot{D}$ name a dense open subset of $2^{<\omega}$ and let $p \in \P$ force this. Let $\theta$ be sufficiently large, $M \prec H_\theta$ countable with $p, \P, \dot{D} \in M$ and let $c$ be an arbitrary Cohen real over $M$. Note that for any $t \in 2^{<\omega}$ the real $c^t:= t^\frown c\hook[ln(t), \omega)$ is also Cohen over $M$. Let $q_c \leq p$ witness the second part of item $(4)$, i.e. $q_c$ forces that $c$ is Cohen over $M[\dot{G}]$. Let $q_c \in G$ be $V$-generic and temporarily work in $V[G]$. Since $q_c$ forced that $c^t$ was Cohen over $M[G]$ for every $t \in 2^{<\omega}$, in particular, we have that $\{c^t \hook k \; | \; k < \omega\} \cap \dot{D}^G \neq \emptyset$ for every $t \in 2^{<\omega}$. Let $f:2^{<\omega} \to \omega$ be a function (defined in $V[G]$) so that for every $t \in 2^{<\omega}$ we have $c^t\hook f(t) \in \dot{D}^G$. Note that since $\dot{D}^G$ is open if $g:2^{<\omega} \to \omega$ is some other function so that for all $t \in 2^{<\omega}$ $g(t) > f(t)$, then $c^t\hook g(t) \in \dot{D}^G$ also. By $\baire$-boundedness there is a function $g \in V$ so that $g:2^{<\omega} \to \omega$ and for all $t \in 2^{<\omega}$ we have $g(t) > f(t), ln(t)$. Fix such a $g \in V$ and let $q' \leq q_c$ in $G$ force this. Back in $V$ we have that $q ' \forces \forall t \in 2^{<\omega} \, \check{c}^t\hook g(t)\check{} \in \dot{D}$. In other words, $q' \forces \{c^t \hook g(t) \; | \; t \in 2^{<\omega}\} \subseteq \dot{D}$. But the set $\{c^t \hook g(t) \; | \; t \in 2^{<\omega}\}$ is dense (and in $V$) so we are done.
\end{proof}

The following is now immediate:

\begin{theorem}\label{thm2}
If $\P$ is Cohen preserving and proper then $\P$ is $\baire$-bounding.
\end{theorem}


\section{The General Preservation Theorem}

In this section we prove Theorem \ref{mainthm11}. We start with recalling the notion of a {\em selective independent family}. The reader familiar with this idea, for example as presented in \cite{DefMIF}, \cite{FM19}, or \cite{CS} can comfortably skip this part and go straight to the proof of Theorem \ref{mainthm11}. Selective independent families were introduced in Shelah's proof of the consistency of $\mfi < \mfu$ in \cite{Sh92}. To facilitate the discussion we utilize the following notation.

\begin{notation} For $\mathcal{I}\subseteq [\omega]^{\omega}$,\begin{itemize}
\item let $\mathrm{FF}(\mathcal{I})$ denote the set of finite partial functions $g$ from $\mathcal{I}$ to $\{0,1\}$, and
\item for $g\in\mathrm{FF}(\mathcal{I})$  write $\mathcal{I}^g$ for $$\bigcap\{A\mid A\in\mathrm{dom}(g)\textnormal{ and }g(A)=1\}\cap\bigcap\{\omega\backslash A\mid A\in\mathrm{dom}(g)\textnormal{ and }g(A)=0\}.$$
\end{itemize}
\end{notation}
In the above notation, a family $\mathcal{I}\subseteq [\omega]^\omega$ is \emph{independent} if $\mathcal{I}^g$ is infinite for all $g\in\mathrm{FF}(\mathcal{I})$. An independent family $\mathcal{I}$ is \emph{maximal} if
$\forall X\in [\omega]^\omega\;\exists g\in\mathrm{FF}(\mathcal{I})$  such that $\mathcal{I}^g\cap X\text{ or }\mathcal{I}^g\backslash X$ is finite. We will need the following strengthening of maximality, which was introduced by Goldstern and Shelah in~\cite{MGSSdm}:

\begin{definition}\label{densemax}
An independent family $\mathcal{I}$ is \emph{densely maximal} if $$\forall X\in [\omega]^\omega\text{ and }g'\in\mathrm{FF}(\mathcal{I})\;\exists g\supseteq g'\text{ in }\mathrm{FF}(\mathcal{I})\text{ such that }\mathcal{I}^g\cap X\text{ or }\mathcal{I}^g\backslash X\text{ is finite.}$$
\end{definition}
In other words, an independent family $\mathcal I$ is {\em densely maximal} if for each $X \in [\omega]^\omega$ the collection of functions $g\in \mathrm{FF}(\mathcal I)$ witnessing that $\mathcal I \cup \{X\}$ is not a larger independent family is dense in $(\mathrm{FF}(\mathcal I), \supseteq)$. 

\begin{definition}\label{density.filter}
Let $\mathcal I$ be an independent family. The \emph{density ideal of $\mathcal{I}$}, denoted $\mathrm{id}(\mathcal{I})$, is $\{X\subseteq\omega\mid \forall g'\in\mathrm{FF}(\mathcal{I})\;\exists g\supseteq g'\text{ in }\mathrm{FF}(\mathcal{I})\text{ such that }\mathcal{I}^g\cap X\text{ is finite}\}$.
Dual to the density ideal of $\mathcal{I}$ is the \emph{density filter of $\mathcal{I}$}, denoted $\mathrm{fil}(\mathcal{I})$ and defined as $$\{X\subseteq\omega\mid \forall g'\in\mathrm{FF}(\mathcal{I})\;\exists g\supseteq g'\text{ in }\mathrm{FF}(\mathcal{I})\text{ such that }\mathcal{I}^g\backslash X\text{ is finite}\}.$$
\end{definition}

For an infinite independent family $\mathcal{I}$, none of the above definitions' meanings change if we replace the word ``finite'' with ``empty''.  For a family $\mathcal X \subseteq [\omega]^\omega$, the set $\langle \mathcal X \rangle_{\mathrm{dn}}$ denotes the downward closure of $\mathcal X$ under $\subseteq^*$, i.e. $A \in \langle \mathcal X \rangle_{\mathrm{dn}}$ if and only if there is an $X \in \mathcal X$ with $A \subseteq^* X$. Similarly, $\langle \mathcal X\rangle_{\mathrm{up}}$ denotes the upward closure of $\mathcal X$ under $\subseteq^*$.  We will make use of the following, see~\cite[Lemma 5.4]{BFS21}.
\begin{lemma}\label{lemma0} A family $\mathcal{I}\subseteq [\omega]^\omega$ is densely maximal if and only if $$P(\omega)=\mathrm{fil}(\mathcal{I})\cup\langle\omega\backslash\mathcal{I}^g\mid g\in\mathrm{FF}(\mathcal{I})\rangle_{\mathrm{dn}}.$$
\end{lemma}

The following are easily verified, see \cite[Lemma 5.5]{BFS21}.
\begin{lemma}\label{lemma1}\hfil
\begin{enumerate}
\item
If $\mathcal I'$ is an independent family and $\mathcal{I}\subseteq\mathcal{I}'$ then $\mathrm{fil}(\mathcal{I})\subseteq\mathrm{fil}(\mathcal{I}')$;
\item If $\kappa$ is a regular uncountable cardinal and $\langle\mathcal{I}_\alpha\mid\alpha<\kappa\rangle$ is a continuous increasing chain of independent families then $\mathrm{fil}(\bigcup_{\alpha<\kappa}\mathcal{I}_\alpha)=\bigcup_{\alpha<\kappa}\mathrm{fil}(\mathcal{I}_\alpha)$;
\item If $\mathcal I$ is an independent family then $\mathrm{fil}(\mathcal{I})=\bigcup\{\,\mathrm{fil}(\mathcal{J})\mid \mathcal{J}\in [\mathcal{I}]^{\leq\omega}\}$.
\end{enumerate}
\end{lemma}

Recall that given a family $\mathcal F$ of subsets of $\omega$ we say that 
\begin{enumerate}
\item
$\mathcal F$ is a $P${\em -set} if every countable $\{A_n \; | \; n < \omega\} \subseteq \mathcal F$ has a pseudointersection $B \in \mathcal F$;
\item
$\mathcal F$ is a $Q$-{\em set} if given every partition of $\omega$ into finite sets $\{I_n \; |\; n < \omega\}$ there is a {\em semiselector} $A \in \mathcal F$,  i.e. a set $A\in\mathcal F$ such that  $|A \cap I_n| \leq 1$ for all $n < \omega$,
\item
$\mathcal F$ is {\em Ramsey} if it is both a $P$-set and a $Q$-set.
\end{enumerate}
If $\mathcal F$ is a filter and a $P$-set (respectively a $Q$-set, Ramsey set) we call $\mathcal F$ a $P$-filter (respectively a $Q$-filter, Ramsey filter).

\begin{definition}\label{def.selective.independent}
An independent family $\mathcal{I}$ is \emph{selective} if it is densely maximal and $\mathrm{fil}(\mathcal{I})$ is Ramsey.
\end{definition}

\begin{fact}[Shelah, see \cite{Sh92}] $\CH$ implies the existence of a selective independent family.
\end{fact}

The main goal of this section is to prove a preservation theorem which greatly generalizes the following list from the literature.

\begin{fact}\label{fact.all}
Let $\mathcal I$ be a selective independent family. Then $\mathcal I$ is remains selective (and hence maximal) after forcing with a countable support product of Sacks forcing  (Shelah, see \cite[Theorem 4.6]{DefMIF} or~\cite[Corollary 37]{FM19}). Moreover, $\mathcal{I}$ remains selective independent after forcing with the countable support iteration of any of the following:
\begin{enumerate}
\item Sacks forcing (Shelah, see \cite[Theorem 4.6]{DefMIF} or \cite[Corollary 37]{FM19});
\item forcing notions of the form $\mathbb Q_\mathcal I$ from Shelah's~\cite{Sh92}; 
\item Miller partition forcing (see~\cite{CFGS21});
\item $h$-Perfect Tree Forcing Notions for different functions $h:\omega \to \omega$ with $1 < h(n) < \omega$ for all $N < \omega$ (see~\cite{CS});
\item coding with perfect trees as in~\cite{BFS21};
\item Miller lite forcing, see Theorem \ref{millerlite2} below;
\item any mix of the above (a consequence of Theorem~\ref{mainthm11}).
\end{enumerate}
\end{fact}

The proofs of all of the above results are similar. However up until now it was not clear how to reunite them into one single  preservation theorem, which is accomplished below:

\begin{theorem}\label{mainthm11}
Let $\delta$ be an ordinal. Let $\mathcal I$ be a selective independent family and let $\langle \P_\alpha \dot{\Q}_\alpha \; | \; \alpha < \delta\rangle$ be a countable support iteration of proper forcing notions so that for every $\alpha < \delta$ we have that $$\forces_\alpha``\dot{\Q}_\alpha\hbox{ is Cohen preserving}".$$ 
If for every $\alpha < \delta$ we have 
$$\forces_\alpha``\dot{\Q}\hbox{ preserves the dense maximality of } \mathcal I"$$ then $\P_\delta$ preserves the selectivity of $\mathcal I$. 
\end{theorem}

The proof of Theorem \ref{mainthm11} is by induction on $\delta$. Before we can begin with it, we need to establish several lemmas. For the rest of this section fix an ordinal $\delta$, a selective independent family $\mathcal I$ and let $\langle \P_\alpha, \dot{\Q}_\alpha \; | \; \alpha < \delta\rangle$ be a countable support iteration of proper forcing notions so that for every $\alpha < \delta$ we have that $\forces_\alpha$``$\dot{\Q}_\alpha$ is Cohen preserving". In particular, by Theorem \ref{Shelah_preservation2}, $\P_\delta$ is Cohen preserving 

\begin{lemma}\label{idealpreserving}
If $G \subseteq \P_\delta$ is generic over $V$ then in $V[G]$,  $\mathrm{id}(\mathcal I)$ is generated by $\mathrm{id}(\mathcal I) \cap V$. 
\end{lemma}
Dually the lemma states that in $V[G]$ the filter $\mathrm{fil}(\mathcal I)$ is generated by $\mathrm{fil}(\mathcal I) \cap V$.

\begin{proof}
It does not matter what $\P_\delta$ is here, just that it is proper and Cohen preserving. Let $p \in \P_\delta$ and let $\dot{X}$ be a $\P_\delta$-name for an infinite subset of $\omega$. Suppose that $p \forces \dot{X} \in {\rm id}(\mathcal I)$. We need to find a ground model $Y \in [\omega]^\omega\cap{\rm id}(\mathcal I)$ and an $r \leq p$, so that $r \forces \dot{X} \subseteq \check{Y}$. Towards this, use the properness of $\P_\delta$ to find a $q \leq p$ and a countable $\mathcal J \subseteq \mathcal I$, so that $q \forces \dot{X} \in {\rm id}(\mathcal J)$. This is possible since every element of the ideal is witnessed by a countable subfamily $\mathcal J \subseteq \mathcal I$ and by the countable covering property of proper forcing notions, such a $\mathcal J$ can be found in $V$. Since $\mathcal J$ is countable, we can associate $\mathrm{FF}(\mathcal J)$ with $2^{<\omega}$. Let $\dot{D}$ be the $\P_\delta$-name for the dense open subset of $2^{<\omega}$ defined by 
$$q \forces \check{g} \in \dot{D}\hbox{ if and only if }\mathcal J^g \cap \dot{X} = \emptyset.$$ 
Since $\P_\delta$ is Cohen preserving, there is an $r \leq q$ and a dense $E \subseteq 2^{<\omega}$ so that $r \forces \check{E} \subseteq \dot{D}$. Let $Y = \bigcap_{h \in E} (\omega \setminus \mathcal J^h)$. Observe that $Y \in {\rm id}(\mathcal J)$ and hence $Y \in {\rm id}(\mathcal I)$. To see this, let $g' \in {\rm FF}(\mathcal J)$ be arbitrary and let $g \supseteq g'$ be in $E$. We have that $Y \subseteq \omega \setminus \mathcal J^g$ and hence $Y \cap \mathcal J^g = \emptyset$ which by definition means that $Y$ is in the ideal. Moreover

\begin{claim}
$r \forces \dot{X} \subseteq \check{Y}$. 
\end{claim}
\begin{proof}
Let $r \in G$ be $\P_\delta$-generic over $V$. Note that by the way $\dot{D}^G$ is defined we have that $\dot{X} = \bigcap_{g \in \dot{D}^G} (\omega \setminus \mathcal J^g)$ and since $E \subseteq \dot{D}^G$ we're done.
\end{proof}
which completes the proof of the Lemma.
\end{proof}

The importance of the above can be summed up by the following:
\begin{lemma}\label{Ramseypreserving}
$\forces_\delta$``${\rm fil}(\mathcal I)$ is Ramsey"
\end{lemma}

\begin{proof}
Since $\P_\delta$ is Cohen preserving, it is $\baire$-bounding and hence ${\rm fil}(\mathcal I)$ remains a Q-set in any generic extension of $V$ by $\P_\delta$. Thus, it remains to see that it is a P-set in any such extension as well. Let $\{\dot{X}_n \; | \; n < \omega\}$ be a countable family of names for elements of ${\rm fil}(\mathcal I)$. By Lemma \ref{idealpreserving}, we know that each $\dot{X}_n$ has a ground model almost subset $Y_n \in {\rm fil}(\mathcal I)$. By properness, there is a superset $\{Z_n \; | \; n < \omega\}$ in the ground model and, since ${\rm fil}(\mathcal I)$ is a P-set in the ground model, there is a pseudo-intersection $Z \in V \cap {\rm fil}(\mathcal I)$. But then $Z$ is a pseudo-intersection of the $\dot{X}_n$'s and so we are done.
\end{proof}

It follows from Lemma \ref{Ramseypreserving} that all we need to do, in order to prove Theorem \ref{mainthm11}, is  to show that dense maximality is preserved. This we will do momentarily though first we need to recall a preservation result due to Shelah (see~\cite[Lemma 3.2]{Sh92}):

\begin{theorem}\label{Shelah_preservation1}
Assume \textsf{CH}. Let $\delta$ be a limit ordinal and let $\langle\mathbb{P}_\alpha,\dot{\mathbb{Q}}_\alpha\mid\alpha<\delta\rangle$ be a countable support iteration of ${^\omega}\omega$-bounding proper posets. Let $\mathcal{F}\subseteq P(\omega)$ be a Ramsey set and let $\mathcal{H}$ be a subset of $P(\omega)\backslash\langle\mathcal{F}\rangle_{\mathrm{up}}$. If $V^{\mathbb{P}_\alpha}\vDash P(\omega)=\langle\mathcal{F}\rangle_{\mathrm{up}}\cup\langle\mathcal{H}\rangle_{\mathrm{dn}}$ for all $\alpha<\delta$ then $V^{\mathbb{P}_\delta}\vDash P(\omega)=\langle\mathcal{F}\rangle_{\mathrm{up}}\cup\langle\mathcal{H}\rangle_{\mathrm{dn}}$ as well.
\end{theorem}

Now we can complete the proof of Theorem \ref{mainthm11}:

\begin{proof}[Proof of \ref{mainthm11}]
The proof is by induction on $\delta$. The successor step is by assumption and so we focus on the limit step. Inductively we have that if $\beta < \delta$ and $G_\beta \subseteq \P_\beta$ is generic over $V$ then $$V[G_\beta] \models P(\omega) = \langle {\rm fil}(\mathcal I) \cap V\rangle_{\rm up} \cup \langle \omega \setminus \mathcal I^g\; | \; g \in \mathsf{FF}(\mathcal I)\rangle_{\rm dn}.$$ But then by Theorem \ref{Shelah_preservation1} plus the fact that ${\rm fil}(\mathcal I)$ is a Ramsey filter in $V^{\P_\delta}$ we get $$\forces_\delta P(\omega) = \langle {\rm fil}(\mathcal I) \cap V\rangle_{\rm up} \cup \langle \omega \setminus \mathcal I^g\; | \; g \in \mathsf{FF}(\mathcal I)\rangle_{\rm dn}.$$
By Lemma \ref{lemma0}, this is exactly what we needed to show.
\end{proof}

\section{An Application: Miller Lite Forcing}

In this section we give an application of Theorem \ref{mainthm11}. Recall from \cite{Ges06} that Miller Lite forcing, denoted $\mathbb{ML}$, consists of finitely branching trees $T \subseteq \omega^{<\omega}$ so that for every $s \in T$ and $n < \omega$ there is a $t \in T$ with $t \supseteq s$ which has at least $n$ many immediate successors. The order is inclusion. We will show the following.

\begin{theorem}\label{millerlite2}
If $\CH$ holds then $\mathbb{ML}$ preserves the dense maximality of any selective independent family. 
\end{theorem}

The forcing notion $\mathbb{ML}$ is proper, in fact Axiom A (\cite[Lemma 6.3]{Ges06}), has the Sacks property (\cite[Lemma 5.1]{OV18}) and hence
by Lemma~\ref{Sacks.implies.Cohen} is Cohen preserving. As a consequence of our preservation theorem we obtain: 

\begin{corollary}
In the Miller Lite model $\mfi = \aleph_1$.
\end{corollary}

Here by ``the Miller Lite model" we mean any model obtained by iteratively forcing with $\mathbb{ML}$ with countable support over a ground model satisfying $\GCH$. The model was first studied by Geschke (see~\cite{Ges06}) in his investigation of continuous colorings of $2^\omega$. Amongst other things, it is shown in~\cite{Ges06} that in the Miller Lite model $\mathfrak{hm} < \mathfrak{l}_{n ,\omega}$ holds for all $n < \omega$\footnote{Since we will not use these cardinals here, and their definitions are a little involved, we omit them and refer the reader to \cite{Ges06}.}. As a result we obtain the following.

\begin{corollary}\label{millerlite2.corollary}
It is consistent that $\mfi = \mathfrak{hm} < \mathfrak{l}_{n, \omega}$ for all $n < \omega$.
\end{corollary}

The rest of this section is devoted to proving Theorem \ref{millerlite2}. We begin with some terminology concerning Miller lite trees. Fix a $p \in \mathbb{ML}$. If $t \in p$ let $p_t = \{s \in p \; | \; s \subseteq t \; {\rm or} \; t \subseteq s\}$ and let ${\rm succ}_p(t)$ denote the set of immediate successors of $t$ in $p$.

\begin{definition}
For $n < \omega$ let 
$$p^n = \{t \in p \; | \; t \in {\rm succ}_p(s)\hbox{ for some }s \in p\hbox{ that is minimal with }|{\rm succ}_p(s)| > n\}.$$ For $p, q \in \mathbb{ML}$ write $p \leq_n q$ if $p \leq q$ and $p^n = q^n$.
\end{definition}

With this terminology, fusion works as usual. Namely if $\{p_n\}_{n < \omega}$ is a sequence of conditions with $p_{n+1} \leq_n p_n$ for all $n < \omega$ then $p_{\omega}:= \bigcap_{n < \omega} p_n$ is a condition again and $p_\omega \leq_n p_n$ for all $n < \omega$. In particular, the sequence of partial orders $\{\leq_n\}_{n < \omega}$ witnesses the aformentioned fact that $\mathbb{ML}$ satisfies Axiom A. Given an $\mathbb{ML}$-name $\dot{X}$ for an infinite set of natural numbers we say that a condition $p$ is {\em preprocessed for} $\dot{X}$ if for each $n < \omega$ and each $t \in p^n$ we have that $p_t$ decides $\dot{X} \cap \check{n}$. A straightforward fusion argument gives the following.

\begin{lemma}\label{preprocessed}
Let $\dot{X}$ be a $\mathbb{ML}$-name for an element of $[\omega]^\omega$. Then the set of conditions preprocessed for $\dot{X}$ is dense in $\mathbb{ML}$. 
\end{lemma}

For a convenience,  for a condition $p \in \mathbb{ML}$ let us denote by ${\rm Split}_k(p)$ the collection of $s \in p$ which are minimal with $|{\rm succ}_p(s)| > k$. Note that this set is finite for any $k$ and $p$.

\begin{proof}[Proof of Theorem \ref{millerlite2}]
Assume $\CH$ and let $\mathcal I$ be a selective independent filter. Since $\mathbb{ML}$ is Cohen preserving the proof of Lemma \ref{Ramseypreserving} applies. Thus, in any generic extension of $V$ by $\mathbb{ML}$ we have that ${\rm fil}(\mathcal I)$ is Ramsey and is generated by the ground model elements. 

Fix a generic $G \subseteq \mathbb{ML}$. We want to show that $\mathcal I$ remains densely maximal in $V[G]$. Let $\dot{X}$ be a $\mathbb{ML}$-name for an infinite subset of $\omega$, so that in $V[G]$ 
for all $g \in {\rm FF}(\mathcal I)$ we have $\dot{X} \nsubseteq \omega \setminus \mathcal I^g$. Let $p \in G$ force this. Without loss of generality we may assume that $p$ is preprocessed for $\dot{X}$. For each $n < \omega$ and each node $t \in p^n$ let $Y_t = \{n \; | \; p_t \nVdash \check{n} \notin \dot{X}\}$. Note that for all $n < \omega$ and all nodes $t \in p^n$ we have that $p_t \forces \dot{X} \subseteq \check{Y}_t$. 

\begin{claim}
For all $n < \omega$ and all nodes $t \in p^n$ we have $Y_t \in {\rm fil}(\mathcal I)$.
\end{claim}

\begin{proof}
Otherwise there are $n < \omega$ and $t \in p^n$ so that $Y_t \subseteq \omega \setminus \mathcal I^g$ for some $g \in {\rm FF}(\mathcal I)$. However $p_t \forces \dot{X} \subseteq \check{Y}_t$ and so $p_t \forces \dot{X} \subseteq \omega \setminus \mathcal I^g$, which  contradicts the choice of $p$.
\end{proof}

Since ${\rm fil}(\mathcal I)$ is a $P$-filter generated by ground model elements, there is a $C \in {\rm fil}(\mathcal I) \cap V$ so that $C \subseteq^* Y_t$ for all $t \in \bigcup_{n < \omega} p^n$. Let $f \in \baire$ be a strictly increasing function such that for all $n < \omega$ we have $C \setminus f(n) \subseteq \bigcap \{Y_t \; | \; t \in p^j, \; j \leq n + 2\}$. The following is proved in \cite{Sh92}, as well as \cite[Lemma 3.15]{CFGS21} and we include it for completeness.

\begin{claim}
There is $C^* \subseteq C$, $C^* \in {\rm fil}(\mathcal I) \cap V$, so that if $\{k_n \; | \; n < \omega\}$ is the strictly increasing enumeration of $C^*$, then $f(k_n) < k_{n+1}$ for all $n < \omega$ and $f(1) < k_1$.
\end{claim}

\begin{proof}
The claim follows from the fact that ${\rm fil}(\mathcal I)$ is a $Q$-filter\footnote{In fact a family $\mathcal F \subseteq [\omega]^\omega$ is a $Q$-filter if and only if for each increasing $f \in \baire$ there is a $C^* = \{k_n \; | \; n < \omega\} \in \mathcal F$ such that $f(k_n) < k_{n+1}$, see \cite[Lemma 3.15]{CFGS21}.}. Indeed.  Inductively find a sequence $\{n_l\}_{l \in \omega}$ so that $n_0 = 0$ and $$n_{l+1} = {\rm min}\{n \; | \; n_l < n {\rm \, and \, for \, all} \, m < n_l \, f(m) < n\}.$$ Consider the interval partition $\mathcal E_0 = \{[n_{3l}, n_{3l + 3})\}_{l \in \omega}$. Since ${\rm fil}(\mathcal I)$ is a $Q$-set and a filter there is a $C_1 \subseteq C$ so that for all $l < \omega$ we have $|C_1 \cap [n_{3l}, n_{3l + 3})| \leq 1$. Now define an equivalence relation $\mathcal E_1$ on $\omega$ by $$m \equiv_{\mathcal E_1} k \, {\rm iff} \, m = k \lor m, k \in C_1 \land (m < k \leq f(m) \lor k < m \leq f(k)).$$
In words, this says that every element of $\omega \setminus C_1$ is in their own equivalence class and distinct elements $m, k \in C_1$ are $\mathcal E_1$-equivalent just in case applying $f$ to the smaller one is greater or equal to the bigger one. Every $\mathcal E_1$ equivalence class has at most two members. To see this, suppose $ m_1 < m_2 < m_3 \in C_1$ were all in the same equivalence class. By definition of $\mathcal E_1$ we have $m_1 <  m_2 < m_3 \leq f(m_1)$. However, since $C_1$ is a semiselector for the interval partition $\mathcal E_0$ there are distinct $l_1 < l_2 < l_3$ so that for all $i \in \{1, 2, 3\}$ we have $m_i \in [n_{3l_i}, n_{3l_{i+1}})$. Thus we get $m_1 < n_{3l_2} \leq m_2  < n_{3l_3} \leq m_3 \leq f(m_1)$ but by the definition of the $n_l$ sequence we also have $f(m_1) \leq n_{3l_2 + 1} < n_{3l_3}$ which is a contradiction.

Now let $C_2 \subseteq C_1$ be a semiselector for $\mathcal E_1$ in ${\rm fil}(\mathcal I)$. Without loss of generality $0 \in C_2$. Let $\{k_n\}$ be an increasing enumeration of $C_2$. For all $n < n'$ we have that $n$ and $n'$ are not in the same $\mathcal E_1$ equivalence class and therefore $f(n) < n'$. As such $C_2 = C^*$ is as needed.

For the final point note that ${\rm fil}(\mathcal I)$ is closed under finite changes to elements so we can augment $C^*$ to get $f(1) < k_1$ as needed.
\end{proof}

We will find a $q \leq p$ forcing that $C^* \subseteq \dot{X}$ and so forcing that $\dot{X}\in{\rm fil}(\mathcal I)$ as desired. We obtain $q$ as a fusion of a sequence. Let $t^* = {\rm Stem}(p)$ be the unique element of ${\rm Split}_0(p)$ and let $p_0 = p = p_{t^*}$. Now for each $(t^*)^\frown i \in {\rm succ}_p(t^*)$ let $w(t^*, i) \in p_0$ be an element of $p_0$ extending $(t^* )^\frown i$. Since $k_1 > f(1)$ we have that $k_1 \in \bigcap \{Y_{w(t^*, i)} \; | \; (t^*)^\frown i \in {\rm succ}_p(t^*)\}$. This means that for each $(t^*)^\frown i \in {\rm succ}_p(t^*)$ there is a $w'(t^*, i) \in p^{k_1 + 1}_0$ extending $w(t^*, i)$ which forces that $k_1 \in \dot{X}$, since $p_0$ is preprocessed. Let $p_1 = \bigcup_{(t^*)^\frown i \in {\rm succ}_p(t^*)} p_{w'(t^*, i)}$. Note that $p_1 \leq_0 p_0$ and forces that $k_1 \in \dot{X}$. Also note that $p^1_1\subseteq p^{k_1 + 1}_0$ and, by construction, for each $m$ we have that $p^m_1 \subseteq p^{k_1 + m}_0$. 
Proceed inductively defining $p_{n+1}$ as follows. Assume $p_n$ has been defined and that for all $m < \omega$ we have $p^{n+m}_n\subseteq p^{k_n + m}_0$. Observe that $k_{n+1} \in \bigcap\{Y_t \; | \; t \in p^n_n\}$, since $k_{n+1} > f(k_n)$. Hence we can find for each $t \in {\rm Split}_n(p_n)$ and each $t^\frown i \in {\rm succ}_{p_n}(t)$ a $w(t, i) \in p^{k_{n+1} + 1}_0$ in $p_n$ which contains $t^\frown i$ so that $(p_n)_{w(t, i)} \forces \check{k}_{n+1} \in \dot{X}$. Let $p_{n+1} = \bigcup_{t \in {\rm Split}_n(p_n)} \bigcup_{t^\frown i \in {\rm succ}_{p_n}(t)} (p_n)_{w(t, i)}$. Finally, let $q$ be the fusion of the 
sequence $\langle p_n:n\in\omega\rangle$. Then $q \forces \check{C}^* \subseteq \dot{X}$ and so 
$q\forces\dot{X}\in{\rm fil}(\mathcal I)$, as desired.
\end{proof}

In \cite{DefMIF} it is shown that if $V=L$ then there is a co-analytic selective independent family (which is the best possible complexity by \cite{Millerpi11}). Thus, by working over $L$ we obtain:
\begin{theorem}\label{thm.miller.lite} 
If $V=L$ then in the model obtained by a countable support iteration of $\mathbb{ML}$ of length $\omega_2$ the witness to $\mfi = \aleph_1$ is co-analytic.
\end{theorem}

\section{Conclusion and Open Questions}
The foregoing leaves a number of interesting questions open:

\begin{question}
Is Cohen preservation needed in Theorem~\ref{mainthm11}? In particular, is it possible that preserving selective independence implies Cohen preservation?
\end{question}

\begin{question}
If $\mathcal I$ is densely maximal and not necessarily selective is it indestructible by some reasonable class of forcing notions?
\end{question}


\begin{question}
Is it consistent there are no selective independent families? Is it possible their (non)existence follows from some cardinal characteristic (in)equality? Are there selective independent families in the Silver model?
\end{question}


\end{document}